\documentclass[12pt,letter,reqno]{amsart}
\usepackage{amsmath,amsthm,amssymb}
\usepackage{lmodern}
\usepackage{fullpage}

\usepackage[usenames,dvipsnames]{xcolor}

\usepackage[T1]{fontenc}
\usepackage{textcomp}
\usepackage{enumitem}
\usepackage[pdfusetitle,
breaklinks=false,pdfborder={0 0
    0},backref=false,colorlinks=true,linkcolor=NavyBlue,citecolor=NavyBlue,urlcolor=black]
{hyperref}
\pdfpageheight\paperheight
\pdfpagewidth\paperwidth
\usepackage[capitalise]{cleveref}

\crefformat{equation}{(#2#1#3)}

\def\Xint#1{\mathchoice
{\XXint\displaystyle\textstyle{#1}}%
{\XXint\textstyle\scriptstyle{#1}}%
{\XXint\scriptstyle\scriptscriptstyle{#1}}%
{\XXint\scriptscriptstyle\scriptscriptstyle{#1}}%
\!\int}
\def\XXint#1#2#3{{\setbox0=\hbox{$#1{#2#3}{\int}$ }
\vcenter{\hbox{$#2#3$ }}\kern-.59\wd0}}

\def\dashint{\Xint-}
\renewcommand{\t}{\tilde}
\newcommand{\mc}[1]{\mathcal{#1}}
\newcommand{\mb}[1]{\mathbb{#1}}
\newcommand{\D}{\mathcal{D}}
\newcommand{\R}{\mathbb{R}}
\newcommand{\C}{\mathbb{C}}
\newcommand{\Z}{\mathbb{Z}}
\newcommand{\T}{\mathbb{T}}
\newcommand{\N}{\mathcal{N}}
\newcommand{\M}{\mathcal{M}}

\newcommand{\W}{\mathcal{W}}
\newcommand{\K}{\mathcal{K}}
\newcommand{\B}{\mathcal{B}}
\renewcommand{\H}{\mathbb{H}}

\renewcommand{\epsilon}{\varepsilon}

\DeclareMathOperator{\dist}{dist}

\DeclareMathOperator{\Tr}{Tr}

\renewcommand{\Im}{\operatorname{Im}}

\DeclareMathOperator{\rank}{rank}
\newcommand{\norm}[1]{\left\|#1\right\|}

\numberwithin{equation}{section}
\newtheorem{theorem}{Theorem}[section]
\newtheorem{lemma}[theorem]{Lemma}
\newtheorem{proposition}[theorem]{Proposition}
\newtheorem{definition}[theorem]{Definition}

\newtheorem{corollary}[theorem]{Corollary}

\theoremstyle{remark}
\newtheorem{remark}[theorem]{Remark}

\newcommand{\cD}{{\mathcal{D}}}

\newcommand{\be}{\begin{eqnarray}}
\newcommand{\ee}{\end{eqnarray}}

\newcommand{\mes}{\mathop{\rm{mes}\, }}

\renewcommand{\mod}{{\rm{mod}\, }}

\def\beeq{\begin{equation}}
\def\eneq{\end{equation}}

\def\bm{\begin{matrix}}
\def\endm{\end{matrix}}

\def\Im{{\rm Im}}

\begin{document}

\title[H\"{o}lder continuity of the integrated density of states]{H\"{o}lder continuity of the integrated density of states for quasi-periodic Jacobi operators}

\author{Kai Tao}
	\address{College of Sciences, Hohai University, 1 Xikang Road Nanjing Jiangsu 210098 P.R.China}
	\email{ktao@hhu.edu.cn,\ tao.nju@gmail.com}

\author{Mircea Voda}
	\address{Department of Mathematics, The University of Chicago
	5734 South University Avenue, Chicago, IL 60615, U.S.A.}
	\email{mircea.voda@math.uchicago.edu}

\thanks{ The first author was supported by the Fundamental Research Funds for the Central Universities (Grant 2013B01014) and the National Nature Science Foundation of China (Grant 11326133, Grant 11401166).}

\date{}

\begin{abstract}
	We show H\"older continuity for the integrated density of states of  a quasi-periodic Jacobi operator
	with analytic coefficients, in the
	regime of positive Lyapunov exponent and with a strong Diophantine condition on  the frequency.
	In particular,
	when the coefficients are trigonometric polynomials we express the H\"older exponent in terms of the
	degrees of the coefficients.
\end{abstract}

\maketitle
\tableofcontents

\section{Introduction}

We consider the quasi-periodic Jacobi operators on $l^2(\Z)$ defined by
\begin{equation*}
	(H(x,\omega)\phi)_n=-b(x+(n+1)\omega)\phi_{n+1}-\overline{b(x+n\omega)}\phi_{n-1}+a(x+n\omega)\phi_n
	,\ n\in\Z,
\end{equation*}
where $ a:\T\to\R $, $ b:\T\to\C $ ($ \T:=\R/\Z $) are real
analytic functions, $ b $ is not identically zero, and $ \omega $ satisfies a strong Diophantine condition.
Specifically, we have
\begin{equation*}
	\omega\in \T_{c,\alpha}:= \left\{ \omega: \norm{n\omega}\ge \frac{c}{n(\log n)^\alpha},~n\ge 1 \right\},
\end{equation*}
with some $ c\ll 1 $ and $ \alpha>1 $.

We let $ H_N(x,\omega) $ be the restriction of $ H(x,\omega) $ to $ [0,N-1] $, with Dirichlet boundary
conditions. We use $ \N(E,\omega) $ and $ L(E,\omega) $ to denote  the integrated
density of states and the Lyapunov exponent  for $ H(x,\omega) $
(see \cref{sec:preliminaries} for definitions).

We will be assuming that $ a $ and $ b $ are trigonometric polynomials of degrees $ d_a $ and $ d_b $.
Let $ d_0:=\max(d_a,d_b) $ and let $ n_b $ be the number of zeroes of $ b $ on $ \T $. Our methods also
apply to general $ a,b $. For the meaning of $ d_0 $ in this general setting see \cref{rem:d_0}. The following is our main result.

\begin{theorem}\label{thm:main}
	Let $ \omega\in \T_{c,\alpha} $ and $ I\subset\R $ be an interval such that
	$ L(E,\omega)>\gamma>0 $ for all $ E\in I $ and let $ p=1/(n_b+2d_0) $.  Fix $ \epsilon>0 $.
	\begin{enumerate}		
		\item
		There exists $ N_0=N_0(a,b,I,\omega,\gamma,\epsilon) $
		such that for any $ N\ge N_0 $, $ (1/N)^{1/p}\ll \eta\le 1/N $, and $ E\in I $ we have
		\begin{equation*}
			\int_\T \left| \sigma(H_N(x,\omega))\cap [E-\eta,E+\eta] \right|\,dx \le N\eta^{p-\epsilon}.
		\end{equation*}
		
		\item
		The integrated density of states satisfies
		\begin{equation*}
			\N(E+\eta,\omega)-\N(E-\eta,\omega)\le \eta^{p-\epsilon},
		\end{equation*}
		for all $ E\in I $ and $ \eta\le \eta_0(a,b,I,\omega,\gamma,\epsilon) $.
	\end{enumerate}
\end{theorem}

Our work generalizes the result of Goldstein and Schlag \cite[Thm. 1.1]{GS-08-Fine} from the Schr\"odinger setting ($ b=1 $). In the almost
Mathieu case ($ b=1 $, $ a(x)=2\lambda\cos(2\pi x) $) the H\"older exponent obtained through this approach is
$ 1/2-\epsilon $, with arbitrary $ \epsilon>0 $.
It is known that the H\"older exponent in this setting cannot be
better than $ 1/2 $ (see for example \cite[Cor. 20]{P-06-A-nonperturbati}), so one gets an asymptotically
optimal result. In fact, Avila and Jitomirskaya \cite{AJ-10-Almost-localiza} showed that the H\"older
exponent is  exactly $ 1/2 $ for the almost Mathieu operator with $ \lambda\neq -1,0,1 $ and general analytic
potentials with small coupling constant.  However, their result covers the positive Lyapunov exponent regime
, via Aubry duality, only for the almost Mathieu operator.

The most important particular example of quasi-periodic Jacobi operator is the extended Harper's model:
\begin{equation*}
	b(x)=\lambda_3 e^{-2\pi i(x+\omega/2)}+\lambda_2+\lambda_1 e^{2\pi i(x+\omega/2)},
	a(x)=2\lambda \cos(2\pi x). 	
\end{equation*}
Unlike the almost Mathieu operator, the positive Lyapunov exponent regime for the extended Harper's model
cannot be approached via duality for all the values of the coupling constants (see \cite{JM-12-Analytic-quasi-}). Therefore, even for this simple operator our result may cover cases not covered by the methods
from \cite{AJ-10-Almost-localiza}.

The main difficulty in extending the work of  Goldstein and Schlag \cite{GS-08-Fine,GS-11-resonances} is
dealing with the singularities coming from the zeroes of $ b $. The groundwork for doing this has been laid in \cite{BV-13-estimate} and \cite{T-11-Hbackslasholder}, where most of the basic tools needed for this paper
have been developed.

The paper is organized as follows. The basic definitions and tools are reviewed in \cref{sec:preliminaries}. The proof of \cref{thm:main} is given in \cref{sec:proof-main}. The proof
relies on the estimate of the number of zeroes for Dirichlet determinants in a small disk, obtained in
\cref{sec:zero-count}. This estimate is obtained through the multiscale method developed in \cref{sec:multiscale-count}. Finally, the auxiliary estimates needed for \cref{sec:multiscale-count} are
established in \cref{sec:harnack}.

\section{Preliminaries}\label{sec:preliminaries}

We begin by recalling the definition of the integrated density of states and some aspects of the transfer
matrix formalism for Jacobi operators.

We use $E_j^{(N)}(x,\omega)$ to denote the eigenvalues of $ H_N(x,\omega) $ and let
\begin{equation*}
\N_N(E,x,\omega) = \frac{1}{N} \left|\left\{ E_j^{(N)}(x,\omega): E_j^{(N)}(x,\omega)<E \right\}\right|.
\end{equation*}
It is known that Kingman's subadditive ergodic theorem implies that there exists $ \N(E,\omega) $ such
that
\begin{equation*}
	\N(E,\omega)=\lim_{N\to \infty} \int_{\T} \N_N(E,x,\omega)\,dx \stackrel{\text{a.s.}}{=}
	\lim_{N\to \infty}  \N_N(E,x,\omega).
\end{equation*}
See for example \cite[Sec. 5.2]{Tes-00-Jacobi}. The quantity $ \N(E,\omega) $ is called the integrated
density of states.

The methods we are using are complex analytic so we will work with an extension of the operator to a
neighbourhood of the real line. We will use the notation
\begin{equation*}
	\H_y:=\{z\in \C: |\Im z|< y\}.
\end{equation*}
It is known that $ a $ and $ b $ admit complex analytic extensions to $ \H_{\rho_0} $ with
$ \rho_0=\rho_0(a,b) $. It is essential for us that $ \det(H_N(\cdot,\omega)-E) $ is a complex
analytic function. To achieve this we need to work with the complex analytic extension of $ \overline b $
instead of $ \overline b $. More precisely, we let $ \t b(z)=\overline{b(\overline z)} $ and we have
\begin{equation}\label{eq:H_N}
	H_N(z,\omega)=
	\begin{bmatrix}
		a\left(z\right) & -b\left(z+\omega\right) & 0 & \ldots & 0\\
		-\t b\left(z+\omega\right) & a\left(z+\omega\right) & -b\left(z+2\omega\right) & \ldots & 0\\
		\ddots & \ddots & \ddots & \ldots & \vdots\\
		0 & \ldots & 0 & -\t b\left(z+\left(N-1\right)\omega\right) & a\left(z+\left(N-1\right)\omega\right)
	\end{bmatrix}.
\end{equation}
The operator is not necessarily self-adjoint off $ \T $, but that would have also been the case if we used
$ \overline b $ instead of $ \t b $ (because the values on the diagonal  are not necessarily real).

We let $M_{N}$ be the $N$-step transfer matrix such that
\begin{equation*}
	\begin{bmatrix}
    	\phi_N\\ \phi_{N-1}
    \end{bmatrix}
    =M_{N}\begin{bmatrix}
    	\phi_0\\ \phi_{-1}
    \end{bmatrix}\, N\ge1.
\end{equation*}
for  any $\phi$ satisfying the difference equation $H\left(z,\omega\right)\phi=E\phi$.
We have that
\begin{equation*}
    M_{N}\left(z,\omega,E\right)=\prod_{j=N-1}^{0}\left(\frac{1}{b\left(z+\left(j+1\right)\omega\right)}
    \left[\begin{array}{cc}
    a\left(z+j\omega\right)-E & -\tilde{b}\left(z+j\omega\right)\\
    b\left(z+\left(j+1\right)\omega\right) & 0
    \end{array}\right]\right),
\end{equation*}
for $z$ such that $\prod_{j=1}^{N}b\left(z+j\omega\right)\neq 0$.
Because $ M_N(z) $ is not necessarily analytic we will in fact work with a version that has the
singularities removed:
\begin{equation*}
	M_{N}^{a}(z,\omega,E)=\left(\prod_{j=1}^{N}b\left(z+j\omega\right)\right)
	M_{N}(z,\omega,E).
\end{equation*}
Based on the definitions, it is straightforward
to check that
\begin{equation}\label{eq:Mu-Ma}
    \log\norm{M_N(z,\omega,E)} =-S_N(z+\omega,\omega)+\log\norm{M_N^a(z)},
\end{equation}
where $S_{N}\left(z,\omega\right)=\sum_{k=0}^{N-1}\log\left|b\left(z+k\omega\right)\right|$. We will
also use $\t S_{N}\left(z,\omega\right)=\sum_{k=0}^{N-1}\log\left|\t b\left(z+k\omega\right)\right|$.
Note that $S_{N}\left(x,\omega\right)=\t S_{N}\left(x,\omega\right)$
for $x\in\mb T$.

We let
\begin{equation*}
	L_{N}\left(y,\omega,E\right)=\frac{1}{N}\int_{\mathbb{T}}\log\left\Vert M_{N}\left(x+iy,\omega,E\right)
	\right\Vert dx,
\end{equation*}
\begin{equation*}
	L\left(y,\omega,E\right)=\lim_{N\rightarrow\infty}L_{N}\left(y,\omega,E\right)=\inf_{N\ge1}L_{N}\left(y,
	\omega,E\right).
\end{equation*}
The limits exist by subadditivity. We also consider the quantities $L_{N}^{a}$ and $L^{a}$ which are defined analogously. Furthermore let $D\left(y\right)=\int_{\mathbb{T}}\log\left|b\left(x+iy\right)\right|dx$.
When $y=0$ we omit the $y$ argument, so for example we write $L\left(\omega,E\right)$
instead of $L\left(0,\omega,E\right)$.  From \cref{eq:Mu-Ma} it follows that
\begin{equation}\label{eq:L-La}
	L\left(\omega,E\right)=-D+L^{a}\left(\omega,E\right).
\end{equation}

Given an interval $ \Lambda=[a,b] $ we let $ H_\Lambda(z,\omega)=H_{b-a+1}(z+a\omega,\omega) $ be the
restriction of $ H(z,\omega) $ to $ \Lambda $ with Dirichlet boundary conditions and
$ f_\Lambda^a(z,\omega,E):=\det(H_\Lambda(z,\omega)-E) $.
A fundamental property of $M_{N}^{a}$ is its relation to the characteristic polynomials of the
finite scale restriction of $ H(x,\omega) $:
\begin{multline}\label{eq:Ma-fa}
    M_{N}^{a}(z)
    =\begin{bmatrix}
    	f_{N}^{a}(z) & -\tilde b(z)f_{N-1}^{a}(z+\omega)\\
	    b(z+N\omega)f_{N-1}^{a}(z) & -\tilde b(z)b(z+N\omega)f_{N-2}^{a}(z+\omega)
    \end{bmatrix}\\
    =\begin{bmatrix}
    	f_{[0,N-1]}^{a}(z) & -\tilde b(z)f_{[1,N-1]}^{a}(z)\\
	    b(z+N\omega)f_{[0,N-2]}^{a}(z) & -\tilde b(z)b(z+N\omega)f_{[1,N-2]}^{a}(z)
    \end{bmatrix}.
\end{multline}
We refer to \cite[Chap. 1]{Tes-00-Jacobi} for a discussion of such relations.

Next we recall some basic tools that will be used throughout the paper. The main tool is a large deviations
estimate for the Dirichlet determinants.

\begin{proposition}\label{prop:ldt-determinant}
	Let $ (\omega,E)\in \T_{c,\alpha}\times\C $ such that $ L(y,\omega,E)>\gamma>0 $,
	$ y\in(-\rho_0,\rho_0) $. For any $ H>0 $, $ N\ge N_0(a,b,E,\omega,\gamma) $, 	and $ |y|<\rho_0 $ we
	have
	\begin{equation*}
		\mes \{ x\in \T: \left| \log|f_N^a(x+iy,\omega,E)|-NL^a(y,\omega,E) \right|
			>H(\log N)^{C_0} \}\le C_1 \exp(-H),
	\end{equation*}
	with $ C_0=C_0(\omega) $ and $ C_1=C_1(a,b,E,\omega,\gamma) $. Furthermore, the same estimate holds for
	all the other entries of $ M_N^a(x+iy,\omega,E) $.
\end{proposition}

\begin{corollary}\label{cor:ldt-determinant}
	Let $ (\omega,E)\in \T_{c,\alpha}\times\C $ such that $ L(\omega,E)>\gamma>0 $.
	For any $ H>0 $, $ N\ge N_0(a,b,E,\omega,\gamma) $, 	and $ |y|\le 1/N $ we
	have
	\begin{equation*}
		\mes \{ x\in \T: \left| \log|f_N^a(x+iy,\omega,E)|-NL^a(\omega,E) \right|
			>H(\log N)^{C_0} \}\le C_1 \exp(-H),
	\end{equation*}
	with $ C_0=C_0(\omega) $ and $ C_1=C_1(a,b,E,\omega,\gamma) $. Furthermore, the same estimate holds for
	all the other entries of $ M_N^a(x+iy,\omega,E) $.
\end{corollary}

The previous two results are slightly modified versions of
\cite[Prop. 2.1]{BV-14-On-optimal-sepa}. We discuss the modifications
in \cref{sec:appendix}. We will only work with \cref{cor:ldt-determinant}, but we need
\cref{prop:ldt-determinant} to justify the following estimate for the integrability of of the entries of
$ M_N^a $.

\begin{corollary}\label{cor:integrability}
	Let $ (\omega,E)\in \T_{c,\alpha}\times\C $ such that $ L(y,\omega,E)>\gamma>0 $,
	$ y\in(-\rho_0,\rho_0) $. There exists a constant
	$ C_0=C_0(a,b,\omega,E,\gamma) $ such that
	\begin{equation*}
		\norm{\log|f_N^a(\cdot,\omega,E)|}_{L^p(\H_{\rho_0})}\le C_0 N p,\ p\ge 1.
	\end{equation*}
	The same estimate hold for all the other entries of $ M_N^a(\cdot,\omega,E) $.
\end{corollary}

We will be interested in the number of zeroes of $ f_N^a $ in a small disk. The reason for this is the
following consequence of the Cartan estimate. See \cref{sec:appendix} for the proof.

\begin{lemma}\label{lem:ldt-Cartan}
	Let $ (\omega,E)\in \T_{c,\alpha}\times\C $ be such that $ L(\omega,E)>\gamma>0 $.
	If $ \zeta_j $, $ j=1,\ldots,k_0 $ are the zeros of $ f_N^a $ in $ \D(z_0,r_0) $
	(counting multiplicities), $ |z_0|\ll 1/N $, $ r_0\ll 1/N $, then
	\begin{equation*}
		\log|f_N^a(z,\omega,E)|>NL^a(\omega,E)-(\log r_0)^2(\log N)^{C_0}+k_0\min_j \log|z-\zeta_j|
			,~z\in \D(z_0,r_0/2),
	\end{equation*}
	with $ C_0=C_0(a,b,E,\omega,\gamma) $,
	provided $ N\ge N_0(a,b,E,\omega,\gamma,k_0) $.
	Furthermore, the same estimate holds for
	all the other entries of $ M_N^a(z,\omega,E) $.
\end{lemma}

The importance of the above result is that it provides an essentially optimal lower bound without any
exceptional set. We will also need the following analogous result for $ b $ and $ \t b $.

\begin{lemma}\label{lem:b-Cartan}
	Let $ \omega\in \T_{c,\alpha} $. If $ \zeta_j $, $ j=1,\ldots,k_0 $ are the zeros of $ b $ in $ \D(z_0,r_0) $
	(counting multiplicities), $ |z_0|\ll 1/N $, $ r_0\ll 1/N $, then
	\begin{equation*}
		\log|b(z)|>D-C_0(\log r_0)^2+k_0\min_j \log|z-\zeta_j|
			,~z\in \D(z_0,r_0/2),
	\end{equation*}
	with $ C_0=C_0(b,\omega) $. Furthermore, the same estimate holds for $ \t b $.
\end{lemma}

It is possible to count the number of zeros of $ f_N^a $ in a small disk via the Jensen formula (see for
example \cite[Sec. 2.3]{Lev-96-Lectures}). Such a straightforward approach yields the following estimate.
We will use the notation
\begin{equation*}
	\nu_f(z_0,r)= \left| \{ z\in \D(z_0,r): f(z)=0 \} \right|.
\end{equation*}
\begin{proposition}\label{prop:zero-count-log}(\cite[Thm. 4.13]{BV-13-estimate})
	Let $ (\omega,E)\in \T_{c,\alpha}\times \C $. There exist constants $ C_0=C_0(a,b,\omega,E,\gamma) $
	and $ N_0=N_0(a,b,\omega,E,\gamma) $ such that
	\begin{equation*}
	 \left|\nu_{f_N^a(\cdot,\omega,E)}(x_0,1/N)\right|\le (\log N)^{C_0},
	\end{equation*}
	for any $ N\ge N_0 $ and $ x_0\in \T $.
\end{proposition}

The proof of the main result hinges on being able to obtain a constant bound on the zeroes, albeit on an
even smaller disk. We will achieve this by using the multiscale counting of zeroes introduced in
\cite[Sec. 9]{GS-11-resonances}. Passing from one scale to the next is done via the Avalanche Principle
(see \cite[Prop. 3.3]{GS-08-Fine}). We will only be using the following particular application of the
Avalanche Principle. We refer to \cite[Cor. 2.7]{BV-14-On-optimal-sepa} for a proof, as the differences
between the results are minor.

\begin{lemma}\label{lem:AP-determinant}
	Let $ (\omega,E)\in\T_{c,\alpha}\times \C $ such that $ L(\omega,E)>\gamma>0 $ and let
	$ A>1$. Let $ \Lambda_j $, $ j=1,\ldots,m $ be pairwise disjoint intervals such that their union
	$ \Lambda $ is also an interval, and $ l\le|\Lambda_j|\le l^A $. Assume that for some
	$ z\in\H_{(2l^A)^{-1}} $
	the large deviations estimate in \cref{prop:ldt-determinant} holds, with some
	$ H\in(0,l(\log l)^{-2C_0}) $,
	 for
	$ f_{\Lambda_j}^a(z,\omega,E) $, $ j=1,\ldots,m $ and $ f_{\Lambda_j\cup\Lambda_{j+1}}^a(z,\omega,E) $,
	$ j=1,\ldots,m-1 $. Then there exists a constant $ l_0(a,b,\omega,E,\gamma,A) $ such that when
	$ l\ge \max(l_0,2\log m/\gamma) $ we have
	\begin{equation*}
		\left| \log|f_\Lambda^a(z)|+\sum_{j=2}^{m-1}\log\norm{A_j(z)}
			- \sum_{j=1}^{m-1}\log\norm{A_{j+1}(z)A_j(z)} \right|\lesssim m\exp(-\gamma l/2),
	\end{equation*}
	where $ A_j(z)=M_{\Lambda_j}^a(z) $, $ j=2,\ldots,m-1 $ and
	\begin{equation*}
		A_1(z)= M_{\Lambda_1}^a(z)\begin{bmatrix}
			1 & 0 \\ 0 & 0
		\end{bmatrix},\quad
		A_m(z)= \begin{bmatrix}
			1 & 0 \\ 0 & 0
		\end{bmatrix}
		M_{\Lambda_m}^a(z).		
	\end{equation*}
	Furthermore, we have
	\begin{equation*}
		\left| \log\norm{M_\Lambda^a(z)}+\sum_{j=2}^{m-1}\log\norm{M_{\Lambda_j}^a(z)}
			- \sum_{j=1}^{m-1}\log\norm{M_{\Lambda_{j+1}}^a(z)M_{\Lambda_j}^a(z)} \right|
			\lesssim m\exp(-\gamma l/2).
	\end{equation*}	
\end{lemma}

It turns out that in conjunction with the Avalanche Principle it is convenient to use the following
double integrals introduced in \cite[Sec. 5]{GS-08-Fine}:
\begin{equation*}
	J_\epsilon(u, z_0, r) =\frac{4}{\epsilon^2}\ \
		\dashint_{\D(z_0, r)}\dashint_{\D(z,\epsilon r)}(u(\zeta)-u(z))\,dA(\zeta)dA(z).
\end{equation*}
We refer to this double integral as a Jensen average. The reason for this is that as a consequence of the
Jensen formula one gets the following estimate.
\begin{lemma} \label{lem:Jensen-avg-analytic}
   	(\cite[Lem. 5.1]{GS-08-Fine})
    Let $f(z)$ be analytic in $\cD(z_0, R_0)$. Then for
    any $ r,\epsilon>0 $ such that $ (1+\epsilon)r<R_0 $ we have
    \begin{equation*}
        \nu_f(z_0, (1-\epsilon)r) \leq  J_\epsilon(\log |f|, z_0,
        r) \leq \nu_f(z_0, (1+\epsilon) r).
    \end{equation*}
\end{lemma}

Finally, we recall the following uniform upper estimates that are essential to the successful use of the
Cartan estimate and the Jensen formula (in conjunction with the deviations estimates).

\begin{proposition}(\cite[Cor. 2.3]{BV-14-On-optimal-sepa})\label{prop:uniform-upper-bound}
	Let $ (\omega_0,E_0)\in \T_{c,\alpha}\times \C $ be such that $ L(\omega_0,E_0)>\gamma>0 $. There exist
	constants $ N_0=N_0(a,b,E_0,\omega_0,\gamma) $, $ C_0=C_0(\omega_0) $, and
	$ C_1=C_1(a,b,E_0,\omega_0,\gamma) $  such that for $ N\ge N_0 $ we have
	\begin{multline*}
		\sup \{ \log\norm{M_N^a(x+iy,\omega,E)}: x\in \T,|E-E_0|,|\omega-\omega_0|\le N^{-C_1},
		|y|\le N^{-1} \}\\ \le NL^a(\omega_0,E_0)+(\log N)^{C_0}.
	\end{multline*}
\end{proposition}

\begin{lemma}\label{lem:uniform-bound-S_N}(\cite[Lem. 2.5]{BV-14-On-optimal-sepa})
Let $ \omega\in \T_{c,\alpha} $. There exist constants $C_{0}=C_{0}\left(\omega\right)$,
$C_{1}=C_{1}\left(b,\omega\right)$ such
that for every $N>1$ we have
\begin{equation*}
	\sup\left\{ S_{N}\left(x+iy,\omega\right):\, x\in\mb T,\,\left|y\right|\le N^{-1}\right\} \le ND
	+C_{1}\left(\log N\right)^{C_{0}}
\end{equation*}
and
\begin{equation*}
	\sup\left\{ \t S_{N}\left(x+iy,\omega\right):\, x\in\mb T,\,\left|y\right|\le N^{-1}\right\} \le ND
	+C_{1}\left(\log N\right)^{C_{0}}.
\end{equation*}
\end{lemma}

\section{Estimates for Jensen Averages}\label{sec:harnack}

For the purposes of the next section we are interested in the Jensen averages of
$ \log\norm{\M_N(z)} $, where $ \M_N(z)=\M_N(z,\omega,E) $ is one of the following matrices:
\begin{equation*}
	M_N^a(z,\omega,E),\quad
	\begin{bmatrix}
		1 & 0\\ 0 & 0
	\end{bmatrix} M_N^a(z,\omega,E),\quad
	M_N^a(z,\omega,E)\begin{bmatrix}
		1 & 0\\ 0 & 0
	\end{bmatrix}.
\end{equation*}
It is to be expected that these Jensen averages are related to the number of zeroes of the entries
of $ \M_N $. In particular we are concerned with the case when the entries have no zeroes and we will show
in \cref{prop:Jensen-estimate-adjusted} that in this case the Jensen average is small.
A straightforward way of controlling these Jensen averages is by estimating the quotients
$ \norm{\M_N(\zeta)}/\norm{\M_N(z)}$, $ \zeta\in \D(z,\epsilon r) $. This will be achieved by using the
Taylor formula in \cref{prop:Taylor-for-monodromy}. The estimate is facilitated by the fact that under the
assumption that the entries of $ \M_N $ have no zeroes we can take advantage of Harnack's inequality.
 We recall a version of Harnack's inequality. This is a minor reformulation
of \cite[Lem. 8.2]{GS-11-resonances}, that doesn't affect its proof.

\begin{lemma}\label{lem:abstract-Harnack}
	Let $ M\gg 1 $, $ r_0 >0 $, $ r_1=(1+\log M)^{-2}r_0 $, $ z_0\in \C $. If  $ f $ is an analytic and
	nonvanishing function on $ \D(z_0,r_0) $ such that
	\begin{equation*}
		\sup_{z\in \D(z_0,r_0)}|f(z)|\le M\text{ and } |f(z_0)|\ge M^{-1},
	\end{equation*}
	then
	\begin{equation*}
		|f(z)|\lesssim |f(z_0)|,~z\in \D(z_0,r_1).
	\end{equation*}
\end{lemma}

In what follows we establish the auxiliary results needed for the proof of \cref{prop:Taylor-for-monodromy}.
\begin{lemma}\label{lem:Cauchy-Harnack}
	Let $ (\omega,E)\in\T_{c,\alpha}\times \C $ be such that $ L(\omega,E)>\gamma>0 $. There
	exists $ N_0(a,b,E,\omega,\gamma) $ such that for any $ k\ge 0 $, $N\ge N_0$, $ |z_0|\ll 1/N $, and
	$ 0<r_0\ll 1/N $ we have
	that if all the entries of $ \M_N(z,\omega,E) $ are either identically zero or have no zeros in
	$ \D(z_0,r_0) $, then
	\begin{equation*}
		\norm{\partial_z^k \M_N(z,\omega,E)}\lesssim k! r_1^{-k} \norm{\M_N(z_0,\omega,E)},
		z\in \D(z_0,r_1),\ r_1=r_0^{1+}.
	\end{equation*}
\end{lemma}
\begin{proof}
	It is convenient for the proof to work with the $ l^1 $ matrix norm. Let $ f_N(z,\omega,E) $ be any of
	the not identically zero entries of $ \M_N(z,\omega,E) $. By \cref{lem:ldt-Cartan} we have
	\begin{equation*}
		\log |f_N(z_0,\omega,E)|\ge NL^a(\omega,E)-(\log r_0)^2(\log N)^C.
	\end{equation*}
	At the same time from \cref{prop:uniform-upper-bound} we know
	\begin{equation*}
		\sup \{ \log|f_N(z,\omega,E)|: z\in \D(z_0,r_0) \}\le NL^a(\omega,E)+(\log N)^C.
	\end{equation*}
	Applying \cref{lem:abstract-Harnack} with $ f=\exp(NL^a)f_N $,
	$ M=\exp((\log r_0)^2(\log N)^C) $ we conclude that
	\begin{equation*}
		\norm{\M_N(z,\omega,E)}\lesssim \norm{\M_N(z_0,\omega,E)}, z\in \D(z_0,r),
	\end{equation*}
	with
	\begin{equation*}
		r=\frac{r_0}{(1+(\log r_0)^2(\log N)^C)^2}\gg r_0^{1+}=r_1,
	\end{equation*}
	provided $ N $ is large enough.
	From the above and the Cauchy formula we get that
	for $ z\in \D(z_0,r_1) $ we have
	\begin{equation*}
		\norm{\partial_z^k \M_N(z,\omega,E)}\lesssim k! r_1^{-k}
			\sup \{ \norm{\M_N(\zeta,\omega,E)}: \zeta \in \D(z_0,2r_1) \}
		\lesssim k! r_1^{-k}\norm{\M_N(z_0,\omega,E)}.
	\end{equation*}
\end{proof}

\begin{lemma}\label{lem:matrix-to-entry}
	If $ B $ is a $ 2\times 2 $ matrix with top-left entry $ b $, then
	\begin{equation*}
		\log \norm{ \begin{bmatrix}
			1 & 0\\ 0 & 0
		\end{bmatrix}+zB}
		=\log|1+bz|+O(|z|^2)\norm{B}^2,\text{ as } z\to 0.
	\end{equation*}
\end{lemma}
For the proof we refer to \cite[p. 835]{GS-08-Fine}. We note that this result is sensitive to the choice of the norm. For example, with the $ l^1 $ norm the error term would be $ O(|z|)\norm{B} $ (we are using the
standard matrix norm induced by the Euclidean norm on $ \C^2 $).

\begin{lemma}\label{lem:mu1-mu2}
	Let $ (\omega,E)\in\T_{c,\alpha}\times \C $ be such that $ L(\omega,E)>\gamma>0 $. There
	exists $ N_0(a,b,E,\omega,\gamma) $ such that for $N\ge N_0$, $ |z_0|\ll 1/N $, and
	$ \exp(-N^{1/2-})\lesssim r_0\ll 1/N $ we have
	that if all the entries of $ \M_N(z,\omega,E) $ are either identically zero or have no zeros in
	$ \D(z_0,r_0) $, then
	\begin{equation*}
		\frac{|\det \M_N(z_0,\omega,E)|}{\norm{\M_N(z_0,\omega,E)}^2}\le \exp(-NL(\omega,E)).
	\end{equation*}
\end{lemma}
\begin{proof}
	We are only concerned with the case $ \M_N=M_N^a $ because the other cases are trivial. Since we have
	\begin{equation*}
		\det M_N^a(z_0,\omega,E)=\exp(\t S_N(z_0,\omega)+S_N(z_0+\omega,\omega)),
	\end{equation*}
	it follows from \cref{lem:uniform-bound-S_N} that
	\begin{equation*}
		|\det M_N^a(z_0,\omega,E)|\le \exp(2ND+(\log N)^C).
	\end{equation*}
	On the other hand, \cref{lem:ldt-Cartan} yields that
	\begin{equation*}
		\norm{M_N^a(z_0,\omega,E)}^2\ge \exp(2NL^a(\omega,E)-(\log r_0)^2(\log N)^C)\ge
		\exp(2NL^a(\omega,E)-N^{1-}).
	\end{equation*}
	The conclusion follows by recalling that we have \cref{eq:L-La}.
\end{proof}

\begin{proposition}\label{prop:Taylor-for-monodromy}
	Let $ (\omega,E)\in\T_{c,\alpha}\times \C $ be such that $ L(\omega,E)>\gamma>0 $. There
	exists $ N_0(a,b,E,\omega,\gamma) $ such that for $N\ge N_0$, $ |z_0|\ll 1/N $, and
	$ \exp(-N^{1/2-})\lesssim r_0\ll 1/N $ we have
	that if all the entries of $ \M_N(z,\omega,E) $ are either identically zero or have no zeros in
	$ \D(z_0,r_0) $, then for $ z\in \D(z_0,r_1^{1+}) $, $ r_1=r_0^{1+} $, we have
	\begin{equation*}
		\log\frac{\norm{\M_N(z,\omega,E)}}{\norm{\M_N(z_0,\omega,E)}}
		=\log|1+b(z-z_0)|+O(|z-z_0|^2 )r_1^{-2}+O(1)\exp(-NL(\omega,E)),
	\end{equation*}
	with $ b=b(z_0) $ and $ |b|\lesssim r_1^{-1} $.
\end{proposition}
\begin{proof}
	Let
	\begin{equation*}
		\M_N(z_0)=U \begin{bmatrix}
			\mu_1 & 0\\ 0 & \mu_2
		\end{bmatrix} V
	\end{equation*}
	be the singular value decomposition of $ \M_N(z_0) $. So, $ U $ and $ V $ are unitary and the singular
	values
	are
	\begin{equation*}
		\mu_1=\norm{\M_N(z_0)}\text{ and } \mu_2=\frac{|\det\M_N(z_0)|}{\norm{\M_N(z_0)}}.
	\end{equation*}
	Using Taylor's theorem, \cref{lem:Cauchy-Harnack}, and \cref{lem:mu1-mu2} we get that for
	$ z\in\D(z_0,r_1) $ we have
	\begin{multline*}
		\frac{\norm{\M_N(z)}}{\norm{\M_N(z_0)}}=\norm{\frac{1}{\mu_1}U^{-1}\M_N(z)V^{-1}}\\
		= \norm{\begin{bmatrix}
			1 & 0\\
			0 & \mu_2/\mu_1
		\end{bmatrix}
		+(z-z_0)B}+O(|z-z_0|^2 )r_1^{-2}\\
		=\norm{\begin{bmatrix}
			1 & 0\\
			0 & 0
		\end{bmatrix}
		+(z-z_0)B}+O(|z-z_0|^2 )r_1^{-2}+O(1)\exp(-NL),
	\end{multline*}
	with
	\begin{equation*}
		\norm{B}\lesssim r_1^{-1}.
	\end{equation*}
	It follows that for $ z\in \D(z_0,r_1^{+}) $ we have
	\begin{equation*}
		\log \frac{\norm{\M_N(z)}}{\norm{\M_N(z_0)}}-\log\norm{\begin{bmatrix}
			1 & 0\\
			0 & 0
		\end{bmatrix}
		+(z-z_0)B}=O(|z-z_0|^2 )r_1^{-2}+O(1)\exp(-NL).
	\end{equation*}
	The conclusion now holds due to \cref{lem:matrix-to-entry}.
\end{proof}

\begin{proposition}\label{prop:Jensen-estimate-adjusted}
	Let $ (\omega,E)\in\T_{c,\alpha}\times \C $ be such that $ L(\omega,E)>\gamma>0 $. There
	exists $ N_0(a,b,E,\omega,\gamma) $ such that for $N\ge N_0$, $ |z_0|\ll 1/N $, and
	$ \exp(-N^{1/2-})\lesssim r_0\ll 1/N $ we have
	that if all the entries of $ \M_N(z,\omega,E) $ are either identically zero or have no zeros in
	$ \D(z_0,r_0) $, then for $ x_0\in\T $, $ \epsilon\in (0,1) $, and $ r\le r_1^{1+}/2 $, $ r_1=r_0^{1+} $,
	we have
	\begin{equation*}
		J_\epsilon(\log\norm{\M_N(\cdot,\omega,E)},x_0,r)=O(r^2)r_1^{-2}+O(1)\epsilon^{-2}
		\exp(-NL(\omega,E)).
	\end{equation*}
\end{proposition}
\begin{proof}
	The result follows from \cref{prop:Taylor-for-monodromy}, the fact that
	\begin{equation*}
		\dashint_{\D(z,\epsilon r)}\log|1+b(\zeta-z)|\,dA(\zeta)=0
	\end{equation*}
	(due to the mean value property for harmonic functions; it is essential that we have
	$ |b|\lesssim r_1^{-1} $ and $ |\zeta-z|\lesssim r_1^{1+} $) and
	\begin{equation*}
		\dashint_{\D(z,\epsilon r)} |\zeta-z|^2 \,dA(\zeta)=\frac{\epsilon^2 r^2}{2}.
	\end{equation*}
\end{proof}

We will also need an estimate for the case when we don't have further information on the entries of
$ \M_N $. For this we use the following result on the Jensen averages of subharmonic functions.

\begin{lemma}\label{lem:Riesz-Jensen}
	(\cite[Lem. 5.4]{GS-08-Fine}) Let
	\begin{equation*}
		u(z)=\int \log|z-\zeta|\mu(d\zeta)+h(z),\ z\in\Omega,
	\end{equation*}
	where $ h $ is harmonic and $ \mu $ is a non-negative measure on some domain $ \Omega $. Then
	\begin{equation*}
		\mu(\D(z_0,(1-\epsilon)r))\le J_\epsilon(u,z_0,r)\le \mu(\D(z_0,(1+\epsilon)r)),
	\end{equation*}
	for any $ z_0,\epsilon,r $ such that $ \D(z_0,(1+\epsilon)r)\subset \Omega $.
\end{lemma}

\begin{proposition}\label{prop:Jensen-estimate-general}
	If $ \M_N(z) $ is analytic on a neighbourhood of the closure of $ \H_{\rho_0} $, then there exists
	$ C_0(a,b,E,\rho_0) $ such that
	\begin{equation*}
		0\le J_\epsilon (\log\norm{\M_N(\cdot,\omega,E)},z_0,r)\le C_0 N,
	\end{equation*}
	for any $ z_0,\epsilon,r $ such that $ \D(z_0,(1+\epsilon)r)\subset \H_{\rho_0} $.
\end{proposition}
\begin{proof}
	Since $ \log\norm{\M_N(z)} $ is subharmonic it admits a Riesz representation:
	\begin{equation*}
		\log{\norm{\M_N(z)}}=\int \log|\zeta-z|\,\mu_N(d\zeta)+h_N(z),
	\end{equation*}
	where $ \mu_N $ is a positive measure and $ h_N $ is harmonic. It is known that
	\begin{equation*}
		\mu_N(\H_{\rho_0})\le CN.
	\end{equation*}
	For a proof we refer to \cite[Lem. 3.4]{BV-13-estimate}. Now the conclusion follows from
	\cref{lem:Riesz-Jensen}.
\end{proof}

\section{Multiscale Counting of Zeroes}\label{sec:multiscale-count}

Given an interval $ \Lambda $ together with a partition into intervals $ \{ \Lambda_j \} $, $ j=1,\ldots,m $
(ordered from leftmost to rightmost) it's easy to see that
\begin{equation*}
	M_\Lambda^a=\prod_{j=m}^1 M_{\Lambda_j}^a.
\end{equation*}
Such a factorization doesn't hold for $ f_\Lambda^a $, but an approximation of this relation is available by
using the Avalanche Principle. This allows one to relate the number of zeroes of $ f_\Lambda $ to that
of $ f_{\Lambda_j} $, $ j=1,\ldots,m $. This is achieved by using Jensen averages and it is
therefore crucial to control the Jensen averages of the extraneous terms that result from the application of
the Avalanche Principle. For this it is natural to introduce the following notion.
\begin{definition}
	We say that $ s\in \Z $ is {\em adjusted} to $ (\D(z_0,r_0),\omega,E) $ at scale $ l $ if for all
	$ l\le k \le 100 l $ and $ |m|\le 100 $ all the entries of $ M_l^a(\cdot+(s+m)\omega,\omega,E) $ have
	no zeros in $ \D(z_0,r_0) $.
\end{definition}
Note that if $ s $ is adjusted then by the results of the previous section we have good control on the
Jensen averages of $ \log\norm{M_{\Lambda'}^a} $, where $ \Lambda' $ can be any interval of size
$ l\le |\Lambda'|\le 100l $ that is ``sufficiently close'' to $ s $. The notion of being adjusted is useful because we can find many adjusted integers.
\begin{lemma}\label{lem:adjust}
	Let $ (\omega,E)\in \T_{c,\alpha}\times \C $, $ x_0\in \T $ and $ n_0\in \Z $. Given $ l\gg 1 $ and
	$ r_0=\exp(-(\log l)^A) $, $ A>1 $, there exists $ n_0'\in[n_0-l^6,n_0+l^6] $ such that $ n_0' $ is adjusted to
	$ (\D(x_0,r_0),\omega,E) $ at scale $ l $.
\end{lemma}
For the proof we refer to \cite[Lem. 9.7]{GS-11-resonances}.

We can now prove the result on multiscale counting of zeroes.
\begin{proposition}\label{prop:multiscale-Jensen}
	Let $ (\omega,E)\in\T_{c,\alpha}\times \C $ such that $ L(\omega,E)>\gamma>0 $ and let
	$ A>1$. Let $ \Lambda_j $, $ j=1,\ldots,m $ be pairwise disjoint intervals such that their union
	$ \Lambda $ is also an interval, and $ l\ll|\Lambda_j|\le l^A $. There exists
	$ l_0=l_0(a,b,\omega,E,\gamma,A) $ such that if $ l\ge \max(l_0,(\log m)^{1+})$
	and all but $ k $ of the intervals
	$ \Lambda_j $ have the endpoints adjusted to $ (\D(x_0,r_0),\omega,E) $ at scale $ l $,
	$ x_0\in \T $, $ \exp(-l^{1/2-})\lesssim r_0\ll 1/l $, then
	\begin{equation*}
		J_\epsilon(\log|f_\Lambda^a|,x_0,r)-\sum_{j=1}^m J_\epsilon(\log|f_{\Lambda_j}^a|,x_0,r)
		=O(1)\epsilon^{-4}r^{-2}\exp(-l^{1-})+(m-k) O(r^2)r_1^{-2}
		+kO(1)C_0l,
	\end{equation*}
	with $ C_0=C_0(a,b,\omega,E,\gamma) $, and for any $ \epsilon\in(0,1) $, $ r\le r_1^{1+}/2 $,
	$ r_1=r_0^{1+} $.
\end{proposition}
\begin{proof}
	The proof is essentially the same as for \cite[Prop. 9.3]{GS-11-resonances}. We partition each
	$ \Lambda_j $ into five intervals $ \Lambda_j^{(i)} $, $ i=1,\ldots,5 $ such that
	$ |\Lambda_j^{(i)}|=l $ for $ i\neq 3 $. Applying the Avalanche Principle expansion to
	$ \log|f_\Lambda^a| $, $ \log|f_{\Lambda_j}^a| $ (i.e. using \cref{lem:AP-determinant} and
	\cref{prop:ldt-determinant}) we get
	\begin{equation*}
		\log|f_\Lambda^a(z)|-\sum_{j=1}^m \log|f_{\Lambda_j}^a(z)|=\sum\pm \log\norm{A_{\Lambda'}(z)}
			+O(1)\exp(-cl),
	\end{equation*}
	for $ z\in \D(z_0,r_0)\setminus \mc B $, $ \mes(\mc B)\le \exp(-l^{1-}) $, with $ A_{\Lambda'}(z) $
	of
	the form
	\begin{equation*}
		M_{\Lambda'}^a(z),\quad \begin{bmatrix}
			1 & 0 \\
			0 & 0
		\end{bmatrix}M_{\Lambda'}^a(z),\text{ or }
		M_{\Lambda'}^a(z)\begin{bmatrix}
			1 & 0 \\
			0 & 0
		\end{bmatrix},
	\end{equation*}
	where $ \Lambda' $ is an interval of length $ l $ or $ 2l $ containing an endpoint of one the intervals
	$ \Lambda_j $. By using \cref{cor:integrability} it follows that
	\begin{equation}\label{eq:Jensen-AP}
		J_\epsilon(\log|f_\Lambda^a|,z_0,r)-
		\sum_{j=1}^m J_\epsilon(\log|f_{\Lambda_j}^a|,z_0,r)
		=\sum\pm J_\epsilon(\log\norm{A_{\Lambda'}},z_0,r)
			+O(1)\epsilon^{-4}r^{-2}\exp(-l^{1-}).
	\end{equation}
	Indeed we have
	\begin{multline*}
		J_\epsilon(\log|f_\Lambda^a|,z_0,r)\\
		=\frac{4}{\pi^2\epsilon^4 r^4}\int_{\D(x_0,r)} \int_{\D(z,\epsilon r)}
			\log|f_\Lambda^a(\zeta)|\,dA(\zeta)\,dA(z)
		- \frac{4}{\pi \epsilon^2 r^2}\int_{\D(x_0,r)}\log|f_\Lambda^a(z)|\,dA(z),
	\end{multline*}
	\begin{multline*}
		\frac{4}{\pi^2\epsilon^4 r^4}\int_{\D(x_0,r)} \int_{\D(z,\epsilon r)\cap \B}
			|\log|f_\Lambda^a(\zeta)||\,dA(\zeta)\,dA(z)\\
		\lesssim \frac{1}{\epsilon^4 r^4}\int_{\D(x_0,r)} Cml^A \sqrt{|\B|}\,dA(z)
		\lesssim \frac{1}{\epsilon^4 r^2}\exp(-l^{1-}),
	\end{multline*}
	and
	\begin{equation*}
		\frac{4}{\pi \epsilon^2 r^2}\int_{\D(x_0,r)\cap \B}|\log|f_\Lambda^a(z)||\,dA(z)
		\lesssim \frac{1}{\epsilon^2 r^2} Cml^A \sqrt{|\B|}
		\lesssim \frac{1}{\epsilon^2 r^2} \exp(-l^{1-}).
	\end{equation*}
	Note that we used the assumption that $ l\ge (\log m)^{1+} $. The other terms are dealt with in the same
	way.
	
	The conclusion follows immediately by applying either \cref{prop:Jensen-estimate-adjusted} or
	\cref{prop:Jensen-estimate-general} to the averages on the right-hand side of \cref{eq:Jensen-AP}.
\end{proof}

\section{Count of Zeroes in a Small Disk}\label{sec:zero-count}

We will show in \cref{prop:zero-count-const} that if $ \Lambda $ has adjusted endpoints then we can use
\cref{prop:multiscale-Jensen} to obtain a bound on the number of zeroes of $ f_\Lambda^a $. The idea is
simply that the zeroes on $ \Lambda $ can be shifted around resulting in more zeroes at a larger scale. The
assumption that $ a,b $ are trigonometric polynomials comes into play via the fact that in this case
$ f_N^a(\cdot,\omega,E) $ is a rational function of degree at most $ 2d_0N $. This is easily seen from
\cref{eq:H_N}.

We will be using the following known results on the equidistribution of the orbit of an irrational shift.
\begin{lemma}\label{lem:discrepancy}
	Let $ \omega\in \T_{c,\alpha} $ and $ N>1 $. There exists a constant
	$ C_0(\omega) $ such that for any interval $ I\subset \T $ we have
	\begin{equation*}
		\# \{ m\in[0,N-1]: m\omega \in I  \}=N|I|+O(1)C_0(\log N)^{\alpha+2}.
	\end{equation*}
\end{lemma}
This lemma is a consequence of the Erd\"os-Tur\'an theorem on the discrepancy of a sequence of
real numbers, and of the Diophantine condition imposed on $ \omega $.
See \cite[Lem. 2.3.2-3]{KN-74-Uniform-distrib} for the resulting estimates for irrational shifts that yield 
the above lemma as a particular case. 
\begin{corollary}\label{cor:equidistribution}
	Let $ \omega\in \T_{c,\alpha} $ and $ N>1 $.
	There exists $ C_0(\omega) $ such that the distance between any
	two consecutive points of the set $ \{ m\omega: m\in[0,N-1] \}\subset \T $ is between
	$ cN^{-1}(\log N)^{-\alpha} $ and $ C_0 N^{-1}(\log N)^{\alpha+2} $.
\end{corollary}
This is an immediate consequence of the previous lemma and the of the Diophantine condition.

\begin{proposition}\label{prop:zero-count-const}
	Let $ (\omega,E)\in\T_{c,\alpha}\times \C $ such that $ L(\omega,E)>\gamma>0 $ and let
	$ A>1$. If the endpoints of $ \Lambda $
	are adjusted to $ (\D(x_0,r_0),\omega,E) $ at scale $ l $, $ \exp(-l^{1/2-})\lesssim r_0\ll 1/l $,
	$ |\Lambda|^{1/A}\le l\ll |\Lambda| $, then $ f_{\Lambda}^a(\cdot,\omega,E) $ has at most $ 2d_0 $ zeroes
	in $ \D(x_0,r_0\exp(-(\log l)^{C_0})) $, for some $ C_0=C_0(a,b,\omega,E,A) $, provided
	$ l\ge l_0(a,b,\omega,E,A) $.
\end{proposition}
\begin{proof}
	Let $ N\simeq\exp(l^{1-}) $. The idea of the proof is that \cref{prop:multiscale-Jensen} implies that if
	$ f_\Lambda^a $ has too many zeroes then $ f_N^a $ has too many zeroes.
	
	\cref{prop:zero-count-log} guarantees that there exists $ n\in [1,(\log l)^C] $ such that
	$ f_\Lambda^a $ has no zeroes in $ \D(x_0,\rho_0)\setminus \D(x_0,\rho_1) $,
	$ \rho_0=r_0^{2+}\exp(-n\log l)$,
	$ \rho_1=\rho_0/l $. Let  $ \Lambda_m=m+\Lambda $, $ x_m=x_0+m\omega $, and
	\begin{equation*}
		S= \{ m\in[0,N-1]: \Lambda_m\subset [l,N-l-1]\text{ and } x_m\in \D(x_0,(1-2\epsilon)\rho_0)  \},
	\end{equation*}
	with $ \epsilon=\epsilon(d_0)\ll 1 $ to be chosen later. Note that \cref{lem:discrepancy} gives us that
	\begin{equation}\label{eq:S-count}
		|S|=2N(1-2\epsilon)\rho_0+O(1)C(\log N)^{\alpha+2}
	\end{equation}
	and due to the Diophantine condition we have that if $ m_1,m_2\in S $, $ m_1\neq m_2 $ then
	\begin{equation*}
		\dist(\Lambda_1,\Lambda_2)\gg l.
	\end{equation*}
	If $ x_m\in \D(x_0,(1-2\epsilon)\rho_0) $ then $ \D(x_0,r_0/2)\subset \D(x_m,r_0) $, because
	$ \rho_0\ll r_0 $. Since we obviously have that the endpoints of $ \Lambda_m $ are adjusted to
	$ (\D(x_m,r_0),\omega,E) $ at scale $ l $ it follows that they are also adjusted to
	$ (\D(x_0,r_0/2),\omega,E) $ at scale $ l $, provided $ m\in S $. It is now easy to see that we can
	find a partition of $ [0,N-1] $ containing the intervals $ \Lambda_m $, $ m\in S $, that satisfies the
	requirements of \cref{prop:multiscale-Jensen} and such that $ 0 $ and $ N-1 $ are the only unadjusted
	endpoints (we are using \cref{lem:adjust}; to make sure that we can apply the lemma, we can replace
	$ r_0 $ by $ r_0\exp(-(\log l)^C) $, as this won't affect the final result).  It then follows that
	\begin{equation}\label{eq:multiscale-ineq}
		\frac{1}{N}J_\epsilon(\log|f_N^a|,x_0,\rho_0)
		\ge \frac{1}{N} \sum_{m\in S} J_\epsilon(\log|f_{\Lambda_m}^a|,x_0,\rho_0)-
		C(\exp(l^{1-})+\rho_0^2(r_0^{1+})^{-2}).
	\end{equation}
	We used the fact that the Jensen averages of subharmonic functions are non-negative (due to the
	sub-mean-value property of subharmonic functions).
	Let $ Z=\nu_{f_\Lambda^a}(x_0,\rho_0) $. We obviously have that
	\begin{equation*}
		Z=\nu_{f_{\Lambda_m}^a}(x_m,\rho_0)=\nu_{f_{\Lambda_m}^a}(x_m,\rho_1),
	\end{equation*}
	for any $ m $. If $ m\in S $ then $ \D(x_m,\rho_1)\subset \D(x_0,(1-\epsilon)\rho_0) $ and therefore
	\begin{equation*}
		\nu_{f_{\Lambda_m}^a}(x_0,(1-\epsilon)\rho_0)\ge Z.
	\end{equation*}
	This, together with \cref{eq:S-count}, \cref{eq:multiscale-ineq}, and \cref{lem:Jensen-avg-analytic}
	imply that
	\begin{equation*}
		\frac{1}{N}\nu_{f_N^a}(x_0,(1+\epsilon)\rho_0)\ge 2(1-2\epsilon)\rho_0Z
		-C(\exp(l^{1-})+\rho_0^2(r_0^{1+})^{-2}).
	\end{equation*}
	
	We can repeat the above reasoning with $ \Lambda_m $ instead of $ \Lambda $, $ x_m $ instead of $ x_0 $,
	and the same $ r_0,\rho_0,\rho_1 $ to get
	\begin{equation*}
		\frac{1}{N}\nu_{f_N^a}(x_m,(1+\epsilon)\rho_0)\ge 2(1-2\epsilon)\rho_0Z
		-C(\exp(l^{1-})+\rho_0^2(r_0^{1+})^{-2}).
	\end{equation*}
	Since we can find at least $ [2\rho_0(1+2\epsilon)]^{-1} $ pairwise disjoint disks
	$ \D(x_m,(1+\epsilon)\rho_0) $ (we are using \cref{cor:equidistribution} and 
	$ (\log N)^{\alpha+2}/N\ll \rho_0 $)
	it follows that
	\begin{equation*}
		2d_0\ge \frac{1}{2\rho_0(1+2\epsilon)} \left(2(1-2\epsilon)\rho_0Z
		-C(\exp(l^{1-})+\rho_0^2(r_0^{1+})^{-2})\right).
	\end{equation*}
	For $ \epsilon=\epsilon(d_0) $ small enough and $ l $ large enough, the above inequality implies that
	$ 2d_0+1>Z $. So we can conclude that $ Z\le 2d_0 $.
\end{proof}

\begin{remark}\label{rem:d_0}
	For general $ a,b $ it follows from the Jensen formula (together with the large deviations estimate and
	the uniform upper bound) that the number of zeroes of $ f_N^a(\cdot,\omega,E) $ in a strip around $ \T $
	is bounded by $ C_0 N $, with $ C_0=C_0(a,b,\omega,E,\gamma) $. It is clear from the proof that in this
	case the previous lemma holds with $ d_0=C_0/2 $.
\end{remark}

\section{Proof of the Main Result}\label{sec:proof-main}

One can get information on the regularity of the integrated density of states from finite scale estimates via
the following standard result.

\begin{lemma}\label{lem:multiscale-IDS}
	For any $ N,m\ge 1 $, $ \omega\in \T $, and any interval $ I\subset \R $ we have
	\begin{equation*}
		\frac{1}{mN}\int_\T |\sigma(H_{mN}(x,\omega))\cap I|\,dx
		\le \frac{1}{N}\int_\T |\sigma(H_N(x,\omega))\cap I|\, dx+\frac{4}{N}.
	\end{equation*}
\end{lemma}
\begin{proof}
	We have that
	\begin{equation*}
		H_{mN}(x)=\bigoplus_{k=0}^{m-1} H_N(x+kN\omega)+R,
	\end{equation*}
	with $ \rank R\le 2m $. It follows from Weyl's interlacing inequalities
	(see \cite[Thm. 4.3.6]{HJ-85-Matrix}) that
	\begin{equation*}
		|\sigma(H_{mN}(x))\cap I|
		\le \sum_{k=0}^{m-1} |\sigma(H_N(x+kN\omega))\cap I|+4m.
	\end{equation*}
	The conclusion follows immediately.
\end{proof}

Let $ \Lambda=[\alpha,\beta] $. The following estimate is well-known from the proof of the Wegner estimate
for the Anderson model:
\begin{multline*}
	|\sigma(H_\Lambda)\cap [E-\eta,E+\eta]|
	\le 2\eta \sum_{j=\alpha}^\beta \frac{\eta}{(E_j^{\Lambda}-E)^2+\eta^2}
	=2\eta\Im \Tr (H_\Lambda-E-i\eta)^{-1}\\
	\le 2\eta \sum_{k=\alpha}^\beta |\langle \delta_k,(H_\Lambda-E-i\eta)^{-1}\delta_k\rangle|.
\end{multline*}
We are left now with finding a bound on the diagonal entries of Green's function. For the Anderson model this
is straightforward using Schur's complement and the independence of the single-site potentials (assuming the
common distribution has bounded density). In the quasi-periodic setting such a simple approach fails due to the
correlations between the single-site potentials. Instead, we will use the fact that due to Cramer's formula we
have
\begin{equation*}
	|\langle \delta_k,(H_\Lambda(x,\omega)-E-i\eta)^{-1}\delta_k\rangle|
	=\frac{ \left| f_{[\alpha,k-1]}^a(x,\omega,E+i\eta) \right|
		\left| f_{[k+1,\beta]}^a(x,\omega,E+i\eta) \right|}
		{ \left| f_{[\alpha,\beta]}^a(x,\omega,E+i\eta) \right| }.
\end{equation*}
We can immediately write an estimate by using the uniform upper bound for the terms on top and the large
deviations theorem for the bottom. This estimate is not of the right order of magnitude, but it can be improved
by using the Avalanche Principle. The idea is simply that if we write the Avalanche Principle expansion for
the determinants, after cancellations, we would be left with a similar quantity but at a much smaller scale.
There are two issues with this approach. First, working with the determinants results in some extra terms that
won't cancel out (namely the  $ A_1, A_m $ terms in \cref{lem:AP-determinant}). Second, $ [\alpha,k-1] $ and
$ [k+1,\beta] $ don't partition $ [\alpha,\beta] $ so we'd be left with some extra terms that we don't want.
These issues are addressed by the following lemma. We will use the notation
\begin{equation*}
	\W_{N,k}(x,\omega,E)
	=\frac{\norm{M_{[0,k-1]}^a(x,\omega,E)}\norm{M_{[k,N-1]}^a(x,\omega,E)}}{\norm{M^a_{[0,N-1]}(x,\omega,E)}}.
\end{equation*}

\begin{lemma}\label{lem:pre-Wegner}
	Let $ (\omega,E)\in \T_{c,\alpha}\times \R $, $ x\in \T $, $ \eta>0 $, $ \K\subset[0,N-1] $, $ N\ge 1 $.
	Then we have
	\begin{equation*}
		|\sigma(H_N(x,\omega))\cap[E-\eta,E+\eta]|
		\le 4\eta \sum_{k\notin \K }\frac{1}{|\t b(x+k\omega)|}\W_{N,k}(x,\omega,E+i\eta)
		+2|\K|+10.
	\end{equation*}
\end{lemma}
\begin{proof}
	We assume that the entry of $ M_{[0,N-1]}^a(x) $ with the largest absolute value is
	\begin{equation*}
		-\t b(x)b(x+N\omega)f^a_{[1,N-2]}(x).
	\end{equation*}
	The case when the largest entry is one of the other entries can be treated analogously to this one. We
	singled out this case because it captures all the needed ideas.
	
	From our assumption we get that
	\begin{equation*}
		\norm{M_{[0,N-1]}^a(x)}\le 2 \left| \t b(x)b(x+N\omega)f^a_{[1,N-2]}(x) \right|.
	\end{equation*}
	To take advantage of this relation we need to work with $ H_{[1,N-2]} $ instead of $ H_{[0,N-1]} $.
	This is not a problem because we have
	\begin{equation*}
		H_{[0,N-1]}=H_{ \{ 0 \} } \oplus H_{[1,N-2]} \oplus H_{ \{ N-1 \} }+R,
	\end{equation*}
	with $ \rank R\le 4 $, and then Weyl's interlacing inequalities (see \cite[Thm. 4.3.6]{HJ-85-Matrix})
	imply
	\begin{multline*}
		|\sigma(H_N(x,\omega))\cap[E-\eta,E+\eta]|
		\le |\sigma(H_{[1,N-2]}(x,\omega))\cap[E-\eta,E+\eta]|+2\rank R+2\\
		\le |\sigma(H_{[1,N-2]}(x,\omega))\cap[E-\eta,E+\eta]|+10.
	\end{multline*}
	We know that
	\begin{equation*}
		|\sigma(H_{[1,N-2]}(x,\omega))\cap[E-\eta,E+\eta]|
		\le 2\eta \sum_{k=1}^{N-2} |\langle \delta_k,(H_{[1,N-2]}(x,\omega)-E-i\eta)^{-1}\delta_k\rangle|.
	\end{equation*}
	We have
	\begin{multline*}
		|\langle \delta_k,(H_{[1,N-2]}(x,\omega)-E-i\eta)^{-1}\delta_k\rangle|
		=\frac{ \left| f_{[1,k-1]}^a(x) \right| \left| f_{[k+1,N-2]}^a(x) \right|}
			{ \left| f^a_{[1,N-1]}(x) \right|}\\
		\le \frac{\norm{M_{[0,k-1]}^a(x)}}{|\t b(x)|}
			\frac{\norm{M_{[k,N-1]}^a(x)}}{|\t b(x+k\omega)||b(x+N\omega)|}
			\frac{2|\t b(x)||b(x+N\omega)|}{\norm{M_{[0,N-1]}^a(x)}}
		= \frac{2}{|\t b(x+k\omega)|}\W_{N,k}(x).
	\end{multline*}
	At the same time we have
	\begin{equation*}
		|\langle \delta_k,(H_{[1,N-2]}(x,\omega)-E-i\eta)^{-1}\delta_k\rangle|
		\le \norm{(H_{[1,N-2]}(x,\omega)-E-i\eta)^{-1}}\le \frac{1}{\eta},
	\end{equation*}
	so we get
	\begin{equation*}
		|\sigma(H_{[1,N-2]}(x,\omega))\cap[E-\eta,E+\eta]|
		\le 4\eta \sum_{k\notin \K}\frac{1}{|\t b(x+k\omega)|}\W_{N,k}(x)+2|\K|,
	\end{equation*}	
	and the conclusion follows immediately.
\end{proof}

We will now see how to estimate $ \W_{N,k} $ by using the Avalanche Principle. Given an interval
$ \Lambda=[\alpha,\beta] $ such that $ 0\in \Lambda $ we will use the notation
\begin{equation*}
	\W_{\Lambda}(x,\omega,E)
	=\frac{\norm{M_{[\alpha,0]}^a(x,\omega,E)} \norm{M_{[1,\beta]}^a(x,\omega,E)}}
		{\norm{M_{[\alpha,\beta]}^a(x,\omega,E)}}.
\end{equation*}

\begin{lemma}\label{lem:multiscale-concatenation-estimate}
	Let $ (\omega,E)\in \T_{c,\alpha}\times \C $ such that $ L(\omega,E)>\gamma>0 $  and . There exists a
	constant $ N_0=N_0(a,b,\omega,E,\gamma) $ such that if $ N\ge N_0 $ and $ \Lambda $ is
	an interval such that $ \Lambda\supset[-|\Lambda|/4,|\Lambda|/4] $,
	$ (\log N)^{1+}\le |\Lambda|\ll N $, then
	\begin{equation*}
		\log|\W_{N,k}(x,\omega,E)|=\log|\W_{\Lambda}(x+(k-1)\omega,\omega,E)|+O(1)\exp(-|\Lambda|^{1-}),
	\end{equation*}
	for $ k\in[2|\Lambda|,N-2|\Lambda|] $ and $ x\in \T\setminus \B_{N,\Lambda}(\omega,E) $,
	with $ |\B_{N,\Lambda}|\le \exp(-|\Lambda|^{1-}) $.
\end{lemma}
\begin{proof}
	Fix $ k\in[2|\Lambda|,N-2|\Lambda|] $. We can partition $ [0,N-1] $ into intervals of size
	proportional to $ |\Lambda| $ (between, say, $ 1/4|\Lambda| $ and $ 4|\Lambda| $) one of which is
	$ (k-1)+\Lambda $. Partitioning $ (k-1)+\Lambda $ as
	\begin{equation*}
		[\alpha+(k-1),k-1]\cup [k,\beta+(k-1)],
	\end{equation*}
	we also induce partitions on $ [0,k-1] $ and $ [k,N-1] $. The conclusion follows by applying the
	Avalanche Principle expansion (i.e. using \cref{lem:AP-determinant} and \cref{prop:ldt-determinant})
	to all three factors in the expression of $ \W_{N,k}(x,\omega,E) $.
\end{proof}

We note that for $ x\in \T\setminus \B_{N,\Lambda} $, with $ \B_{N,\Lambda} $ as in the previous lemma, we
have
\begin{equation*}
	\log\norm{M_{(k-1)+\Lambda}^a(x)}\ge \log|f_{(k-1)+\Lambda}^a(x)|\ge |\Lambda|L^a-|\Lambda|^{1-}.
\end{equation*}
This, together with the uniform upper bound from \cref{prop:uniform-upper-bound}, imply that
\begin{equation*}
	\left| \W_{\Lambda}(x+(k-1)\omega) \right|\le \exp(|\Lambda|^{1-}).
\end{equation*}
Such an estimate is not good enough. It will be clear that we  need  $ (\log|\Lambda|)^C $ instead of
$ |\Lambda|^{1-} $. While it is certainly possible to apply the large deviations estimate with a deviation
of size $ (\log|\Lambda|)^C $, the resulting exceptional set would be too large for the Avalanche Principle
and also for bounding the integral of $ |\sigma(H_N)\cap [E-\eta,E+\eta]| $ over it. This difficulty will be
overcome by using \cref{lem:ldt-Cartan}.

We will also use the following standard estimate.
\begin{lemma}\label{lem:p-sum-of-shifts}
	Let $ \omega\in \T_{c,\alpha} $ and $ p>1 $. There exists  a constant $ C_0(\omega,p) $ such that
	for any $ N>1 $ and $ \rho \gg 1/N $ we have
	\begin{equation*}
		\sum_{k\in S} \norm{k\omega}^{-p}\le C_0 N(\log N)^\alpha \rho^{1-p},
	\end{equation*}
	where
	\begin{equation*}
		S= \{ k\in[0,N-1]: \norm{k\omega} \ge \rho \}.
	\end{equation*}
\end{lemma}
\begin{proof}
	Let $ x_1\le \ldots \le x_n $ be the elements of the set $ \{ k\omega (\mod 1): k\in S \} $ and
	$ x_0=x_1-1/N $. Note that we have $ x_0\ge \rho/2 $. Also, due to the Diophantine restriction on
	$ \omega $ we have $ x_{i+1}-x_{i}\ge C N^{-1}(\log N)^{-\alpha} $. We can now conclude that
	\begin{multline*}
		\sum_{k\in S} \norm{k\omega}^{-p}=\sum_{i=1}^n x_i^{-p}
		\le \sum_{i=1}^n \frac{1}{x_i-x_{i-1}}\int_{x_{i-1}}^{x_i} t^{-p}\,dt\\
		\le C N(\log N)^\alpha \int_{\rho/2}^1 t^{-p}\,dt\le C' N(\log N)^\alpha \rho^{1-p}.
	\end{multline*}
\end{proof}

\begin{proof}
	(of \cref{thm:main}) We just have to prove the first part of the theorem. The second part follows from
	the first and \cref{lem:multiscale-IDS}.
	
	Let $ l=(\log N)^2 $, $ r_0=\exp(-(\log l)^2) $, and $ r_1=r_0\exp(-(\log l)^{C_0}) $, with $ C_0 $ as in
	\cref{prop:zero-count-const}, with the given $ r_0 $ and $ A=10 $. Let $ \{ x_j \}
	$ be a minimal set of points such that the disks $ \D(x_j,r_1/2) $ cover $ \T $. By \cref{lem:adjust} we
	can find intervals $ \Lambda_j=[\alpha_j,\beta_j] $, $ \alpha_j\simeq -l^7 $, $ \beta_j\simeq l^7 $, such
	that $ \alpha_j,\beta_j $ are adjusted to $ (\D(x_j,r_0),\omega,E) $ at scale $ l $. It follows from
	\cref{prop:zero-count-const} that $ f_{\Lambda_j}^a(\cdot,\omega,E) $ has at most $ 2d_0 $ zeroes in
	$ \D(x_j,r_1) $. Furthermore, for $ N $ large enough, $ \t b $ has at most $ n_b $ zeroes in
	$ \D(x_j+\omega,r_1) $. Therefore, from \cref{lem:ldt-Cartan}, \cref{prop:uniform-upper-bound}, and
	\cref{lem:b-Cartan} we get
	\begin{multline*}
		\left |\frac{1}{\tilde b(x+\omega)}\W_{\Lambda_j}(x)\right|
		\le \exp((\log l)^C) |x-\zeta_j|^{-2d_0}|x+\omega-\zeta_j'|^{-n_b}\\
		\le \exp((\log l)^C) \max\left(|x-\zeta_j|^{-(2d_0+n_b)},|x+\omega-\zeta_j'|^{-(2d_0+n_b)}\right)
	\end{multline*}
	for all $ x\in \D(x_j,r_1/2) $.

	Let $ \B=\cup \B_{N,\Lambda_j} $, with $ \B_{N,\Lambda_j} $ as in
	\cref{lem:multiscale-concatenation-estimate} and $ \K $ be the set of integers $ k $ that are not in
	$ [l^8,N-l^8] $ (i.e., to which we cannot apply \cref{lem:multiscale-concatenation-estimate}), such
	that $ x+(k-1)\omega $ is at distance less than $ \rho_0 $ from the zeroes of $ f_{\Lambda_j}^a $ in
	$ \D(x_j,r_1) $, or such that $ x+k\omega $ is at distance at least $ \rho_0 $ from the zeroes of
	$ \t b $, with $ \rho_0\gg 1/N $ to be chosen later. We have that $ |\B|\le \exp(-(\log N)^{14-}) $ and
	\begin{equation*}
		|\K|\lesssim d_0N \exp((\log l)^C)\rho_0+n_bN\rho_0+(\log N)^C.
	\end{equation*}

	Applying \cref{lem:pre-Wegner} and \cref{lem:multiscale-concatenation-estimate} we get that for
	$ x\in\T\setminus \B $ be have
	\begin{multline}\label{eq:eta-rho-bound}
		|\sigma(H_N(x,\omega))\cap[E-\eta,E+\eta]|
		\lesssim \eta \sum_{k\notin \K }\frac{1}{|\t b(x+k\omega)|}\W_{N,k}(x,\omega,E+i\eta)
		+|\K|\\
		\lesssim N\eta\exp((\log l)^C)\rho_0^{1-(2d_0+n_b)}+|\K|.
	\end{multline}
	We obtained the $ \rho_0^{1-(2d_0+n_b)} $ factor instead of a $ \rho_0^{-(2d_0+n_b)} $ factor by using
	\cref{lem:p-sum-of-shifts} (this is the reason for needing $ \rho_0\gg 1/N $).
	At this point we are essentially looking for a choice of $ \rho_0 $ such that
	\begin{equation*}
		\eta \rho_0^{1-(2d_0+n_b)}+\rho_0\lesssim \eta^p,
	\end{equation*}
	with $ p $ as large as possible. An elementary analysis yields that the largest possible H\"older exponent is
	$ p=1/(2d_0+n_b) $ and it is attained when $ \rho_0=\eta^p $. Now we get that for any
	$ (1/N)^{1/p}\ll \eta\le 1/N  $ (in fact, for the upper bound all we need is that
	$ \eta^{0+}\exp((\log l)^C)\le 1 $) we have
	\begin{equation*}
		|\sigma(H_N(x,\omega))\cap[E-\eta,E+\eta]|\le N\eta^{p-},
	\end{equation*}
	for any $ x\in \T\setminus \B $. Note that for \cref{eq:eta-rho-bound} to hold we need to ensure that
	$ L(E+i\eta,\omega)\gtrsim \gamma $. This is true for $ N $ large enough, by continuity of the Lyapunov
	exponent (see \cite{JM-12-Analytic-quasi-}). Since for any $ x\in\T $ we have
	\begin{equation*}
		|\sigma(H_N(x,\omega))\cap[E-\eta,E+\eta]|\le N
	\end{equation*}
	and $ |\B|\le \exp(-(\log N)^{14-}) $ it follows that
	\begin{equation*}
		\int_\T |\sigma(H_N(x,\omega))\cap[E-\eta,E+\eta]|\,dx \lesssim  N \eta ^{p-}.
	\end{equation*}
	
	Finally, let us note that to obtain the first part by using \cref{lem:multiscale-IDS} one needs that
	$ \eta^{p-} \gtrsim N^{-1}  $, which is not a problem.
\end{proof}

\appendix

\section{Discussion of some Results from \cref{sec:preliminaries}}\label{sec:appendix}

First we discuss \cref{prop:ldt-determinant} and \cref{cor:ldt-determinant}. \cref{prop:ldt-determinant} for
the determinants is just \cite[Prop. 2.1]{BV-14-On-optimal-sepa} stated for general $ y $ instead of just
$ y=0 $. This is fine because the large deviations estimate depends only on the positivity of the Lyapunov
exponent. In particular, the fact that the operator is Hermitian for $ y=0 $ is not used. The statement for
the other entries follows from the estimate for $ f_N^a $. It is clear from \cref{eq:Ma-fa}
that one needs to control the deviations of $ b $ and $ \t b $. This is easily achieved by applying the
large deviations estimate for subharmonic functions \cite[Thm. 3.8]{GS-01-Holder}. To get
\cref{cor:ldt-determinant} we simply use the fact that
\begin{equation*}
	|NL^a(y,\omega,E)-NL^a(\omega,E)|\le C (N|y|+(\log N)^2).
\end{equation*}
This follows from the estimates
\begin{equation*}
	0\le L_N^a(y,\omega,E)-L^a(y,\omega,E)< C\frac{(\log N)^2}{N}
\end{equation*}
and
\begin{equation*}
	|L_N^a(y,\omega,E)-L_N^a(\omega,E)|\le C|y|
\end{equation*}
which were established in \cite[Lem. 3.9, Cor. 3.13]{BV-13-estimate}.

Next we prove \cref{lem:ldt-Cartan}. We will use the following formulation of Cartan's estimate (cf. \cite[Thm. 11.4]{Lev-96-Lectures} and \cite[Lem. 2.4]{GS-11-resonances}).

\begin{lemma}\label{lem:Cartan}
	Let $ \phi $ be an analytic function on $ \D(z_0,r_0) $, $ z_0\in\C $ and let $ m,M $ be such that
	\begin{equation*}
		\sup_{\D(z_0,r_0)} \log|\phi(z)|\le M,\quad m\le \log|\phi(z_0)|.
	\end{equation*}	
	Given $ H\gg 1 $, there exists
	\begin{equation*}
		\mc B=\mathop{\bigcup}_{j=1}^K \D(z_j,r_j),\ K\lesssim H(M-m),\ \sum_{j=1}^K r_j\le r_0\exp(-H),
	\end{equation*}
	such that
	\begin{equation*}
		\log|\phi(z)|-M\gtrsim H(M-m),
	\end{equation*}
	for $ z\in \D(z_0,r_0/6)\setminus \mc B $.
\end{lemma}
\begin{proof}
	(of \cref{lem:ldt-Cartan}) From \cref{cor:ldt-determinant} with
	$ H=-C\log r_0 $, $ C\gg 1 $ we know that there exists $ z_1 $, $ |z_1-z_0|\ll r_0 $ such that
	\begin{equation*}
		\log|f_N^a(z_1)|>NL^a+(\log r_0)(\log N)^C.
	\end{equation*}
	We can now apply Cartan's estimate on $ \D(z_1,100 r_0) $, with
	\begin{equation*}
		H=-C\log r_0,\quad M=NL^a+(\log N)^C,\quad m=NL^a+(\log r_0)(\log N)^C,
	\end{equation*}
	to get that
	\begin{equation*}
		\log|f_N^a(z)|>NL^a-(\log r_0)^2(\log N)^C,
	\end{equation*}
	for $ z\in \D(z_0,r_0)\setminus \mc B $, with $ \mc B $ as in \cref{lem:Cartan}. We can guarantee that
	there
	exists $ r\in(r_0/2,r_0) $ such that $ \partial\D(z_0,r)\subset \D(z_0,r_0)\setminus \mc B $ and
	\begin{equation*}
		\min_j \dist(\zeta_j,\partial \D(z_0,r))\gtrsim \frac{r_0}{k_0+1}.
	\end{equation*}
	The minimum principle now implies that
	\begin{equation*}
		\log \left| \frac{f_N^a(z)}{\prod(z-\zeta_j)} \right|>NL^a-(\log r_0)^2(\log N)^C
			+k_0\log c\frac{r_0}{k_0+1}>NL^a-2(\log r_0)^2(\log N)^C,
	\end{equation*}
	for $ z\in \D(z_0,r) $. The conclusion follows immediately.
\end{proof}
Finally, we note that \cref{lem:b-Cartan} follows analogously by using the large deviations estimate for
subharmonic functions \cite[Thm. 3.8]{GS-01-Holder}.

\bibliographystyle{alpha}
\bibliography{../Schroedinger}

\end{document}